\documentclass[]{amsart}
\usepackage{epsfig}
\usepackage{amsmath}
\usepackage{amssymb}
\usepackage{cite}
\usepackage{color}
\numberwithin{equation}{section}

\newtheorem{theorem}{Theorem}[section]
\newtheorem{lemma}[theorem]{Lemma}
\newtheorem{corollary}[theorem]{Corollary}
\newtheorem{proposition}[theorem]{Proposition}

\newtheorem{definition}{Definition}[section]

\newtheorem{remark}[theorem]{Remark}

\allowdisplaybreaks \numberwithin{equation}{section}

\newcommand{\id}{\mathrm{id}}
\newcommand{\Hom}{\mathrm{Hom}}
\newcommand{\diag}{\mathrm{diag}}

\newcommand{\sfa}{\mathsf{a}}

\newcommand{\sff}{\mathsf{f}}
\newcommand{\sfh}{\mathsf{h}}

\newcommand{\weglassen}[1]{}

\renewcommand{\imath}{\mathrm{i}}

\DeclareMathOperator{\charOp}{char}
\DeclareMathOperator{\divOp}{div}
\DeclareMathOperator{\modOp}{mod}
\DeclareMathOperator{\ord}{ord}
\DeclareMathOperator{\rank}{rank}

\title{$p$-adic  $\lambda$ functions for cyclic Mumford curves}
\author{Yaacov Kopeliovich}
\address{University of Connecticut, Department of Finance, 2100 HillSide Road, Storrs CT, 06269}
\email{yaacov.kopeliovich@uconn.edu}
\thanks{Part of this research was partially conducted in Advanced Research Workshop: Isogeny based post-quantum cryptography, 
held at Hebrew University of Jerusalem, July 29-31, 2024, 
funded by a grant by NATO-Science for Peace and Security (SPS) 
Advanced Research Workshops (ARW).
The author is thankful to the organizers Tony Shaska,
and Shaul Zemel for an invitation and funding}
\thanks{This note wouldn't have been written had it not been for Jeremy Teitelbaum. I thank him for posing the question and his encouragement to pursue it.}

\begin{document}
\maketitle{}

\begin{abstract}We express the branch points cross ratio of cyclic Mumford curves as quotients 
of $p$-adic theta functions evaluated at the $p$-adic period matrix
\end{abstract}

\section{introduction}
In this note we generalize $\lambda$ functions of cyclic curves that are known in the complex case, see for example \cite{FK980}, to the $p$-adic case. That is, given a $p$-adic Mumford cyclic curve of genus $g$ with the equation 
\begin{equation}
y^p=x^{r_0}\prod_{i=1}^{n}(x-\lambda_i)^{r_i}
\end{equation} 
for $p$ primes, we derive formulas expressing cross ratios of $\lambda_i$  for the $p$-adic case  as quotients of $p$-adic theta functions.  
The question of obtaining such formulas is a natural question following from \cite{T}. In \cite{T}, 
$p$-adic periods of modular curves $X_0(p)$ are calculated for certain primes $p$ 
such that the genus of $X_0(p)$ equals 2. The calculation of periods involves the following 3 steps: 
\begin{enumerate} 
\item In the case of a $g=2$ curve with branch points  at $0$, $1$, $\infty$, $\lambda_1$, $\lambda_2$, $\lambda_3$, 
 formulas for $\lambda_i$ expressing them as ratios of products of $p$-adic theta functions are obtained.  
\item These formulas are used to express  $\lambda_i$ as $p$-adic infinite series in entries of the $p$-adic period matrix. \label {first}
\item These series are inverted to obtain expressions for the periods as functions of $\lambda_1$, $\lambda_2$,
$\lambda_3$.
\item The exceptional zero conjecture for these $X_0(p)$ is verified.
\end{enumerate} 

We  focus on (\ref{first}) and derive formulas for $\lambda_1$, \ldots, $\lambda_{r}$ for an arbitrary $g$. 
This generalizes the formulas from \cite{T}   to cyclic curves of any genus. 
In \cite{T},  the formulas for $\lambda_i$ are derived by analyzing  the fundamental $p$-adic group action on the associated tree and its fundamental domain. Our analysis will follow the classical approach that was exemplified in \cite{FK980} and \cite{M2}. 
This approach enables us to write the quotients of theta functions as ratios of $\lambda_i$ if 
the corresponding divisors are non-special.  In the cyclic case,
we characterize these non-special divisors  supported on the branch points of the covering. 
Thus, the only task that is left is to compute the $p$-adic characteristics of the images of divisors. 
We accomplish this task following the work of \cite{V1,V2}. 

This paper  has  three sections. In the first section, we collect the relevant facts about Mumford curves and $p$-adic theta functions. In the second section, we calculate the images of the relevant characteristics of the Jacobian variety. 
In the last section, we prove our main theorem.

\section{Non-Special Divisors }
In this section, our goal is to give a dimension formula for  sections into line bundles that are supported on divisors of its ramification points. We work with the general case of cyclic covers. 
Let $\mathbb{K}$ be an algebraically closed field and $\charOp (\mathbb{K})=0$.
\begin{definition}\label{cycliccover}
A $p$ cyclic covering of $\mathbb{P}^{1}(K)$ is a Galois covering 
$\phi: X \to \mathbb{P}^{1}(\mathbb{K})$ of degree $p$ where $X$ 
is a non-singular projective curve over $\mathbb{K}$ with genus $g$.
\end{definition}
If $F(x,z)$ is the function field of $X,$ the equation that defines it as an extension of  $F(x)$ is
$z^p= M(x)/N(x)$  or: 
\begin{equation}
(zN(x))^p=M(x)N(x)^{p-1}.
\end{equation}
Therefore, we assume that the equation for such a cover is of the form
\begin{equation}\label{XEqNoInfBP}
y^p=\prod_{i=1}^{n+1}(x-a_i)^{r_i},
\end{equation}
and the genus equals 
$$g(X)= \tfrac{1}{2} (n-1)(p-1).$$ 

In the equation \eqref{XEqNoInfBP} we assumed that all the branch points are finite. 
If we  work with a model in which $\infty$ is going to be one of the branch points then the equation becomes 
\begin{equation}
y^p=\prod_{i=1}^{n}(x-a_i)^{r_i}.
\end{equation}
Now, $\infty$ and $a_i$, $i\,{=}\,1$,\ldots, $n$, are branch points of the cover. 
Let $B_i=\phi^{-1}(a_i)$, $1\leq i\leq n$, and $B_{\infty}=\phi^{-1}(\infty)$. Then
$$\divOp(y)=\sum_{i=1}^{n}r_i B_i- \bigg(\sum_{i=1}^{n} r_i \bigg)B_{\infty}.$$

\begin{theorem}\label{dimensionformula}
Let $D\,{=}\,\sum_{i=1}^n d_i D_i$ be a divisor supported on finite branch points, 
and $d_j \,{\in}\, \{0,\dots,p-1\}$. 
Define the set  $\{g_k\}_{k=0}^{p-1}$ by the equations 
\begin{equation}
\frac{1}{p}\bigg(\deg D-\sum_{j=1}^n \overline{k r_j + d_j}
+\overline{\sum_{j=1}^n k r_j}\bigg) +1=g_k,\quad   0\leq k\leq p-1,
\end{equation}
where $\overline{r}$ denotes the remainder of the division of $r$ by $p$. Then 
\begin{equation}
\dim H^{0}(X,\mathcal{O}(-D))=\sum_{k=0}^{p-1} \max(0,g_k).
\end{equation}
\end{theorem}
\begin{proof}
The curve $X$ has a cyclic automorphism, given explicitly through 
$$T:(x,y)\mapsto (x,\zeta y),$$ 
where $\zeta$ is a $p$-th root of unity. 
$D$ is invariant under the action of $T$. Thus,  $T$ acts on $H^{0}(X,\mathcal{O}(-D))$.
We conclude that 
$$H^{0}(X,\mathcal{O}(-D)) \,{=}\,\bigoplus_{\chi\in \mathbb{Z}_p^{\ast}} V_{\chi},$$
where $V_{\chi}$ is the eigenspace of $T$ with a character 
$\chi\,{:}\,{\mathbb{Z}_p} \,{\to}\, \mathbb{K}$. 
Let $n_{\chi}\,{=}\,\dim V_{\chi}$. Clearly, 
$$\dim H^{0}(X,\mathcal{O}(-D)) = \sum_{\chi\in \mathbb{Z}_p^{\ast}} n_{\chi}.$$ 
We  find $n_{\chi}$ for all $\chi$.

 The characters $\chi$ of $T$ are of the form $\zeta^{\ell}$ for some $\ell$,
$0\,{\leq}\, \ell \,{\leq}\, p\,{-}\,1$. Hence, each $\chi$ is assigned with 
$\ell$ such that $\forall f \,{\in}\, V_{\chi}$ $f/y^{\ell}$ has a trivial action of $T$ on it. 
Thus, $f/y^{\ell}$ is a pullback of a function from $\mathbb{P}^{1}(\mathbb{K})$.
Let $H_{\ell}^{0}$ be the subspace of $H^0(X,\mathcal{O}(-D-\divOp(y^{\ell})))$ 
composed of all pullbacks from functions on $\mathbb{P}^{1}(\mathbb{K}).$
\begin{lemma}
$\exists D_{\ell}$, a divisor on $\mathbb{P}^{1}(\mathbb{K})$, 
such that $H^{0}(\mathbb{P}^{1}(\mathbb{K}),\mathcal{O}(-D_{\ell}))$ is isomorphic to $H^{0}_{\ell}$.
\end{lemma}
\begin{proof}
If $B_j$ is a ramification point appearing in the divisor $D$
with multiplicity $d_j$, 
let $m_j\,{=}\,[(\ell r_j\,{+}\,d_j)/p ]\,{\times}\, p$, that is,  $m_j$ is the maximal number such that 
$m_j \,{\leq}\, \ell r_j\,{+}\,d_j$ and $m_j \,{=}\, 0 \modOp p$. 
Let  $Q_{a_j}$ represent the point that corresponds 
to $a_j$ in $\mathbb{P}^{1}(\mathbb{K})$. Define
\begin{equation}
D_{\ell} = \sum_{j=1}^n \frac{m_j}{p} Q_{\lambda_j}-
\bigg[\ell \frac{\sum_{j=1}^n r_j}{p}\bigg] \infty.
\end{equation} 
We show that $D_{\ell}$ is the desired divisor. Let $h:\mathbb{P}^{1}(\mathbb{K})\to \mathbb{P}^{1}(\mathbb{K})$ 
such that $\divOp (h) \,{\geq}\, {-} D_{\ell}$.  By $\widehat{h}$ we denote a lift of $h$, 
and so $\widehat{h} \,{\in}\, H_{\ell}^{0}$.  Then
\begin{equation}
-\divOp (\widehat{h})\geq \sum_{j=1}^n -\tilde{d}_j B_j+ \ell \sum_{j=1}^n  r_i\infty.
\end{equation}
We require $\tilde{d}_j \,{\leq}\, \ell r_j \,{+}\, d_j$. 

Conversely, if $\widehat{h} \in H_{\ell}^{0}$ and $\widehat{h}$ is a lift of $h$, 
then by definition $$\ord_{P_j} \widehat{h} \geq - (\ell r_j + d_j).$$ 
We  have  $\ord_{P_j} \widehat{h}  \,{=}\,0\modOp p$,
since $\widehat{h}$ is a lift of $h$. 
We conclude, that $\ord_{Q_j} h \,{\geq}\, m_j/p$, 
and hence, $h\in H_{\ell}^{0}$, as required. 
\end{proof}
Now we compute $\deg D_{\ell}$: 
\begin{equation}
\deg D_{\ell} =\frac{1}{p}\bigg(\sum_{j=1}^{n} m_j \bigg)
-\bigg[\frac{\sum_{j=1}^n \ell r_j}{p}\bigg]. 
\end{equation}
By definition, $m_j\,{=}\, \ell r_j \,{+}\, d_j \,{-}\,\overline{\ell r_j \,{+}\, d_j}$. Hence, we can write
\begin{equation}
\begin{split}
\deg D_{\ell} &= \frac{1}{p}\bigg(\sum_{j=1}^n (\ell r_j + d_j)
-\sum_{j=1}^n \overline{\ell r_j + d_j}-\sum_{j=1}^n \ell r_j+\overline{\sum_{j=1}^n \ell r_j}\bigg) \\
&=\frac{1}{p}\bigg(\deg D -\sum_{j=1}^n\overline{\ell r_j + d_j}+\overline{\sum_{j=1}^n \ell r_j}\bigg).
\end{split}
\end{equation}
Applying the Riemann-Roch theorem to $D_{\ell}$, we obtain
$$\dim H^0 \big(\mathbb{P}^{1}(\mathbb{K}),\mathcal{O}(-D_{\ell}) \big)
= \max (\deg D_{\ell} +1,0). $$ 
Therefore,
\begin{equation}
\begin{split}
\dim H^0 \big(X,\mathcal O(-D)\big) &= \sum_{\chi_{\ell} \in Z_p^{*}} \max(\deg D_{\ell} +1,0)\\
&=\sum_{\ell=0}^{p-1} \max\bigg(\frac{1}{p}\bigg(\deg D - \sum_{j=1}^n \overline{\ell r_j + d_j}
+\overline{\sum_{j=1}^n \ell r_j}\bigg)+1,0\bigg).
\end{split}
\end{equation}
The last sum equals 
$\sum_{\ell=0}^{p-1} \max(0,g_{\ell})$, which completes the proof of the theorem. 
\end{proof}

\section{Cyclic Mumford Curves}\label{firstsection} 
In this section, we collect the relevant facts about cyclic Mumford curves 
over non-Archimedean fields. The reader may consult \cite{GV,V1,V2} and the references therein for proofs of the assertions presented in this section. 
Let $\mathbb{K}$ be a complete non-Archimedean field 
which is algebraically closed and $\charOp (\mathbb{K})\,{=}\,0.$

Let $\Gamma$ be a Schottky group of rank $g$, that is, $\Gamma$ is a free discontinuous subgroup of 
$\mathrm{PGL}(2,\mathbb{K})$. Let $\Omega\,{=}\,\{x \,{\in}\, \mathbb{P}^{1}(\mathbb{K}) \mid 
x \text{ is an ordinary point for } \Gamma \}$. 
$\Gamma$ acts freely on $\Omega$.
\begin{definition}
Let $X_{\Gamma}\,{=}\,\Omega/\Gamma$ be the Mumford curve associated with $\Gamma$. 
\end{definition}
We assume that $\infty$ is an ordinary point. Let $N(\Gamma)$ 
be the normalizer of $\Gamma$ in $\mathrm{PGL}(2,\mathbb{K})$. 
The following is shown in \cite{V1,V2}.
\begin{proposition}\label{Pgamma0}
$X_{\Gamma}$ admits a $p$-cyclic covering of $\mathbb{P}^{1}(\mathbb{K})$ 
if and only if there exists $\sigma_0\,{\in}\, N(\Gamma)$ such that: 
\begin{enumerate}
\item $\sigma_0^p= \id$,
\item $\forall \gamma \in \Gamma$\ \ 
$\prod_{k=0}^{p-1} \sigma_0^k \gamma \sigma_0^{-k}=1\modOp [\Gamma,\Gamma]$.
\end{enumerate}
\end{proposition}
 Let $\Gamma_0 \,{=}\,\langle \Gamma,\sigma_0\rangle$. 
The theorem from \cite{DS75} characterizing the action of periodic automorphism on 
 free groups enables \cite{V2} to show\footnote{ see Proposition 2.2. in \cite{V2} and the discussion preceding it}: 
\begin{theorem}\label{T:sGenIntro}
There exist $\sigma_0$, \ldots, $\sigma_s \,{\in}\, \Gamma_0$ with the following properties: 
\begin{enumerate}
\item $\sigma_i^p = \id$, for $i=1$, \ldots, $s$,
\item $\Gamma_0 = \prod_{i=1}^s \langle \sigma_i\rangle $ is a free product,
\item Let $\xi_{j,1} \,{=}\, \sigma_j \sigma_0^{-1}$ and  $\xi_{j,i} \,{=}\, \sigma_0 \xi_{j,i-1} \sigma_0^{-1}$ for 
$j=1$, \ldots, $s$, $i=2$, \ldots, $p-1$. Then $\xi_{j,i}$ form a free basis of $\Gamma$.
\end{enumerate}
\end{theorem}
The natural mapping $\phi:\Omega/\Gamma \,{\to}\, \Omega/\Gamma_0$ is a realization of the mapping 
$X_\Gamma \,{\to}\, \mathbb{P}^{1}(\mathbb{K})$, since  $\rank \Gamma_0\,{=}\,0$.
The fact that $\Gamma_0$ is a free product of elements with finite order implies 
that $\Omega/\Gamma_0$ is isomorphic to $\mathbb{P}^{1}(\mathbb{K})$.


\subsection{Analytical Jacobian Variety of Mumford curves}
\begin{definition}\label{ThetaABDef}
For each $a,b \,{\in}\, \Omega$ we define the function 
$\Theta_{a,b}\,{:}\, \Omega \,{\to}\, \mathbb{K}$ as
follows
\begin{equation} 
\Theta_{a,b}(z)=\prod_{\gamma \in \Gamma}\frac{z-\gamma a}{z-\gamma b}.
\end{equation} 
\end{definition}
Let $\mathcal{O}_a$ denote the orbit of $a$ under the action of $\Gamma$.
If $a\in \Omega$, then $\mathcal{O}_a \,{\subset}\, \Omega$.
\begin{theorem}\label{T:ThetaAB}
The following properties are satisfied by $\Theta_{a,b}$:
\begin{enumerate}
\item $\Theta_{a,b}$ is a convergent product and satisfies the equation  
\begin{equation}\label{OmegaGamma}
\forall \gamma \in \Gamma\ \ \forall z \in \Omega
\qquad c_{a,b}(\gamma) \Theta_{a,b}(z)=\Theta_{a,b}(\gamma z).
\end{equation}
\item If $b \,{\not \in}\, \mathcal{O}_a$, 
then $\Theta_{a,b}$  has zeros precisely on $\mathcal{O}_a$ and poles on $\mathcal{O}_b$, 
and has no other zeros or poles in $\Omega$.
\item If $b \,{\in}\, \mathcal{O}_a$, then $\Theta_{a,b}$ does not depend on $a$. 
Let $b\,{=}\,\beta a$, then we assign  $c_{\beta}(\alpha) \,{=}\,c_{a,b}(\alpha)$, 
and denote $\Theta_{a,b}$ by $\Theta_{\beta}$.
\end{enumerate}
\end{theorem}
For the proof, 
see \cite[Chap.\,2]{GV}. 

Let $G_{\Gamma} \,{=}\,\Hom(\Gamma,\mathbb{K}^{\ast})$. 
Because $\Gamma$ is free we can identify $G_{\Gamma}$ with $\left({\mathbb{K}^{\ast}}\right)^{g}$. 
The group $A_{\Gamma} \,{=}\, \left\{c_{\gamma} \mid \gamma\in \Gamma\right\}$ 
is a free Abelian group of rank $g$, which is discrete subgroup in $G_{\Gamma}$. 
\begin{definition}
$G_{\Gamma}/A_{\Gamma}$ is the analytical Jacobian variety of the Mumford curve $X_{\Gamma}$. 
\end{definition}
Assume, $D=\sum_{i=1}^m (a_i-b_i)$ is a divisor on $X_{\Gamma}$ of degree $0.$ $D$ is principal 
if it is a divisor of function $f:X_{\Gamma}\to P^1(\mathbb{K})$.
\begin{definition}
If $X_{\Gamma}$ is an algebraic curve its Jacobian variety $J_{\Gamma}$ 
is the group of divisors of degree $0$ modulo the subgroup of principal divisors. 
\end{definition}

Then we have, see  \cite{GV},
\begin{theorem}
The mapping $\sum_{i=1}^m (a_i-b_i)\mapsto \prod_{i=1}^m c_{a_i,b_i}$ induces an analytical isomorphism 
 between $J_{\Gamma}$ and $G_{\Gamma}/A_{\Gamma}$. 
\end{theorem}
\begin{remark}
This theorem is a multiplicative version of the usual mapping 
between algebraic and analytical  versions of the Jacobian variety, expressed  through Abelian integrals.  
\end{remark}
For the proof, see \cite{GV}.
\begin{definition}
Choose a basepoint $o \in \Omega$ and define  
$u_{o}\,{:}\,X_{\Gamma} \,{\to}\, G_{\Gamma}$ by $u_{o}(x)\,{=}\,c_{x,o}$. 
Extend the mapping $u_{o}$ to divisors of any degree on 
$X_{\Gamma}$ using the multiplicative property of $c_{a,b}$.
\end{definition}

\subsection{$p$-Adic Theta Functions}

We define the action of  $A_{\Gamma}$  on 
$$\mathcal{O}^{\ast}(G_{\Gamma})=\{ \sff  \mid \sff \text{ is holomorphic and non-vanishing function on } G_{\Gamma} \}$$ 
by $\sff_{c_\beta}(c)= \sff(c_\beta c)$. 
If  $\sfa_{c}{}\in Z^{1}(A_{\Gamma},G_{\Gamma})$ 
is a 1-cocycle, we denote 
$$\mathcal{L}_{\sfa} = \big\{\sfh \mid \sfh \text{ is a holomorphic function on } 
G_{\Gamma}, \forall c_\gamma \,{\in}\, A_{\Gamma}\ \sfh(c)\,{=}\,\sfa_{c_{\gamma}}(c) \sfh(c_\gamma c) \big\}.$$ 
\begin{definition} 
Elements of $\mathcal{L}_{\sfa}$ are called theta functions of type $\sfa$.
\end{definition} 
Now, we construct the basic theta function. 
Let $\rho\,{:}\,A_{\Gamma} \,{\times}\, A_{\Gamma} \,{\to}\, \mathbb{K}^{\ast}$ 
be a symmetric bilinear form such that 
$\rho^2(c_\beta,c_\gamma) = c_\beta(\gamma)$ for all $\beta,\gamma \in \Gamma$. 
Define a canonical 1-cocycle by 
$$\sfa_{c_\gamma}(c)= \rho(c_\gamma,c_\gamma) c(\gamma),\quad
c_\gamma\in A_{\Gamma},\ \ c\in G_{\Gamma}.$$ 
In this case, $\dim \mathcal{L}_{\sfa} =1$,  and $\mathcal{L}_{\sfa}$ is generated by the Riemann theta function
\begin{equation} 
\theta_{\Gamma}(c)=\sum_{c_\gamma \in A_{\Gamma}} \sfa_{c_\gamma}(c).
\end{equation}
The divisor of $\theta_{\Gamma}$ is  invariant under the action of $A_{\Gamma}$, and hence, induces a divisor on $J_{\Gamma}$. This divisor defines a polarization on $J_{\Gamma}$. The $p$-adic version of the Riemann 
vanishing theorem is valid, i.e.  we have 
\begin{theorem}[\textbf{Riemann Vanishing Theorem}]\label{RVT}\
\begin{enumerate} 
\item The holomorphic function $\theta_{\Gamma} \circ u_o$ 
has a $\Gamma$-invariant divisor of degree $g$, 
which is regarded as a divisor on $X_{\Gamma}$.
\item If $K_{\Gamma} \,{=}\, 
(\divOp (\theta_{\Gamma} \circ u_o) - o) \modOp \Gamma\,{\in}\, \divOp(X_{\Gamma})$, 
then $2 K_{\Gamma}$ is a canonical divisor. 
The class of $K_{\Gamma}$ under linear equivalence of divisors 
does not depend on the choice of the basepoint $o$.
\item Let $c\,{\in}\, G_{\Gamma}$. Then $\theta_{\Gamma}(c)=0$ 
if and only if $c\,{=}\,u_o(D -K_{\Gamma})$ for some positive divisor $D$ of degree $g$. The order of vanishing of $u_o$ at $c$ equals $i(D)$, the index of specialty of $D$. 
\end{enumerate} 
\end{theorem} 
For the proof, see  \cite{V2,GV}.

\section{Images of branch points  in the Jacobian variety} 
Our goal in this section is to calculate images of  branch points  
on the Jacobian variety generalizing the calculation carried out in \cite{V1,V2} for cyclic covers. 
We continue with the assumptions of Section\;\ref{firstsection}.
We start with the following lemma from \cite{V2}.
\begin{lemma}\label{L:Theta}
Assume $\alpha \,{\in}\, N(\Gamma)$, then $\forall a,b\in \Omega$  we have
$$\Theta_{a,b}(\alpha z)=c_{a,b}(\alpha)\Theta_{\alpha^{-1}a,\alpha^{-1}b}(z).$$
\end{lemma}
\begin{proof} Let $\alpha$ has the form
$$\alpha=
   \begin{pmatrix} 
      q & t \\ 
      v & w 
   \end{pmatrix},
$$
and the action of $\alpha$ on  $\mathbb{P}^{1}(\mathbb{K})$ is given by
$$ \alpha z = \frac{qz+t}{vz+w}.$$
We assume that $\det\alpha=1$, then 
$$\alpha^{-1}=
   \begin{pmatrix} 
      w & -t \\ 
      -v & q 
   \end{pmatrix}.
$$
Then, we compute
\begin{equation} \label{a}
\Theta_{a,b}(\alpha z)=\prod_{\gamma \in \Gamma}\frac{\alpha z - \gamma a}{\alpha z- \gamma b}
= \prod_{\gamma \in \Gamma}\frac{qz+t - vz\gamma a - w\gamma a}{qz+t - vz \gamma b - w\gamma b }.
\end{equation}   
Taking into account that
$$z-\alpha^{-1}\gamma a = z - \frac{w\gamma a-t}{-v\gamma a+q}
=\frac{-v\gamma a z + qz - w\gamma a+t}{-v\gamma a+q},$$ 
we rewrite   the product in \eqref{a} as 
\begin{equation}
\prod_{\gamma \in \Gamma}\frac{z-\alpha^{-1}\gamma a }{z-\alpha^{-1} \gamma b }
\prod_{\gamma \in \Gamma}\frac{-v\gamma a+q}{-v\gamma b+q}.
\end{equation}
Let
\begin{equation}\label{cabDef}
c_{a,b}(\alpha) = \prod_{\gamma \in \Gamma} \frac{-v\gamma a+q}{-v\gamma b+q}.
\end{equation}  
Thus,  
$$\prod_{\gamma \in \Gamma} \frac{\alpha z -\gamma a}{\alpha z -\gamma b}
= c_{a,b}(\alpha) \prod_{\gamma \in \Gamma}
\frac{z-\alpha^{-1}\gamma a }{z-\alpha^{-1} \gamma b }
= c_{a,b}(\alpha) \prod_{\tilde{\gamma} \in \Gamma}
\frac{z - \tilde{\gamma} \alpha^{-1} a }{z - \tilde{\gamma} \alpha^{-1}  b },$$  
where $\tilde{\gamma} = \alpha^{-1}\gamma \alpha$.
This concludes the proof. 
 \end{proof}
 As a corrolary we have: 
 \begin{lemma}\label{L:cc}
$\forall \alpha \in N(\Gamma)\ \ \forall \gamma\in \Gamma\ \ 
c_{\alpha^{-1}a, \alpha^{-1}b} (\gamma) = c_{a,b}(\alpha \gamma \alpha^{-1})$.
\end{lemma}
\begin{proof}
By Lemma\;\ref{L:Theta},  and then Theorem\;\ref{T:ThetaAB} we have
\begin{equation}\label{A1}
\Theta_{a,b}(\alpha \gamma z) = c_{a,b}(\alpha) \Theta_{\alpha^{-1}a, \alpha^{-1}b}( \gamma z)
= c_{a,b}(\alpha) c_{\alpha^{-1}a, \alpha^{-1}b} (\gamma) 
\Theta_{\alpha^{-1}a, \alpha ^{-1} b}(z).
\end{equation}
At the same time, due to $\alpha \gamma \alpha^{-1} \,{\in}\,\Gamma$ we have
\begin{multline}\label{A2}
\Theta_{a,b}(\alpha \gamma z) = \Theta_{a,b}(\alpha \gamma \alpha^{-1} \alpha z)
= c_{a,b}(\alpha \gamma \alpha^{-1}) \Theta_{a,b}(\alpha z) \\
= c_{a,b}(\alpha \gamma \alpha^{-1}) c_{a,b}(\alpha) \Theta_{\alpha^{-1}a, \alpha^{-1}b}(z).
\end{multline}
Comparing \eqref{A1} and \eqref{A2} we obtain the result.
\end{proof}

Our goal is to calculate $c_{\sigma_i}$ explicitly for the generators of $\Gamma_0$
introduced in Theorem\;\ref{T:sGenIntro}.

We have $\sigma_i\in \mathrm{PGL}(2,\mathbb{K}),$ and $\mathbb{K}$ is an algebraically closed field. 
Hence, $\sigma_i$ are diagonalizable.  That is, for each $\sigma_i$ there exists $\upsilon_i$ such that  
$\sigma_i \,{=}\,\upsilon_i^{-1} \beta_i \upsilon_i$, 
$\beta_i=\diag(\zeta^{\nu_i }, 1)$
and $\zeta$ is a $p$-th root of unity. Assume that $o_i$, $e_i$ are ordinary fixed points of $\sigma_i$. 
Since the fixed points of $\beta_i$ are  $o_i\,{=}\,0$ and $e_i\,{=}\, \infty$, we have
$\upsilon_i o_i =  0$ and $\upsilon_i e_i = \infty$. 
\begin{lemma}\label{ThetaSigma}
Let $a$ be an arbitrary point in $\Omega$, and $e_i$ be a
fixed point of $\sigma_i$ such that $e_i \,{=}\, \upsilon_i^{-1} \infty$. Then: 
\begin{equation}
\Theta_{a,e_i}(\sigma_i z) = \zeta^{\nu_i}\Theta_{\sigma_i^{-1}a, e_i}(z).
\end{equation}
\end{lemma}
\begin{proof}
First, we prove the lemma when $\sigma_i$ is diagonal, that is $\sigma_i\,{=}\,\beta_i$.  Then we can write
\begin{multline}\label{diag}
\Theta_{a,\infty}(\beta_i z) = \prod_{\gamma \in \Gamma}
\frac{\beta_i z-\gamma a}{\beta_i z- \gamma \infty}
=(\zeta^{\nu_i}z-a)\prod_{\gamma \neq \id,\gamma\in \Gamma}
\frac{\zeta^{\nu_i}z-\gamma a}{\zeta^{\nu_i}z-\gamma \infty}\\
=\zeta^{\nu_i}(z-\zeta^{-\nu_i}a) \prod_{\gamma \neq \id, \gamma\in \Gamma}
\frac{z-\zeta^{-\nu_i}\gamma a}{z-\zeta^{-\nu_i}\gamma \infty}
=\zeta^{\nu_i}\prod_{\gamma \in \Gamma}
\frac{z-(\beta_i^{-1}\gamma \beta_i) \beta_i^{-1} a}{z-(\beta_i^{-1}\gamma \beta_i) \beta_i^{-1} \infty} \\
=\zeta^{\nu_i}\prod_{\tilde{\gamma} \in \Gamma}
\frac{z-\tilde{\gamma} \beta_i^{-1} a}{z-\tilde{\gamma} \infty}
= \zeta^{\nu_i}\Theta_{\beta_i^{-1}a,\infty}(z).
\end{multline}

For a generic $\sigma_i$ we have $\upsilon_i \sigma_i \upsilon_i^{-1}\,{=}\, \beta_i$, which is a diagonal matrix. 
Let $\tilde{z} \,{=}\, \upsilon_i z$, then $\beta_i \tilde{z} \,{=}\, \upsilon_i \sigma_i z$.
Using \eqref{diag} and Lemma\,\ref{L:Theta}, we find
$$ c_{\upsilon_i a,\infty} (\upsilon_i) \Theta_{a, e_i}(\sigma_i z)  = \Theta_{\upsilon_i a,\infty}(\beta_i \tilde{z})  =  
\zeta^{\nu_i}\Theta_{\beta_i^{-1} \upsilon_i a,\infty}(\tilde{z}) 
=  c_{\beta_i^{-1} \upsilon_i a, \infty}(\upsilon_i)\, \zeta^{\nu_i} \Theta_{ \sigma_i^{-1} a, e_i}(z).$$
By a direct computation with the help of \eqref{cabDef} one can see
that $c_{\upsilon_i a,\infty} (\upsilon_i) \,{=}\, c_{\beta_i^{-1} \upsilon_i a,\infty}(\upsilon_i)$.
This concludes the proof.
\end{proof}

\begin{corollary}\label{ch}
$\forall \gamma \,{\in}\, \Gamma\ $ $\Theta_{\gamma} (\sigma_i z)=\zeta^{-\nu_i}\Theta_{\sigma_i^{-1} \gamma \sigma_i}(z)$.
\end{corollary}
\begin{proof}
By definition of $\Theta_{\gamma}(\sigma_i z) =\Theta_{e_i,\gamma e_i}(\sigma_iz).$  Apply Lemma\;\ref{ThetaSigma}: 
\begin{equation}
\Theta_{e_i,\gamma e_i}(\sigma_iz) = \Theta_{\gamma e_i, e_i}^{-1}(\sigma_i z)=
\zeta^{-\nu_i}\Theta^{-1}_{\sigma_i^{-1}\gamma \sigma_i\sigma_i^{-1}e_i, \sigma_i^{-1}e_i}(z)\\=\zeta^{-\nu_i}\Theta_{\sigma_i^{-1} \gamma \sigma_i}(z).
\end{equation}
 \end{proof}
We also have:
\begin{corollary}
$\forall \gamma\in \Gamma \ $ $c_{\sigma_i^k a,e_i}(\gamma) = c_{a,e_i}(\sigma_i^{-k}\gamma \sigma_i^{k})$.
\end{corollary}
\begin{proof}
The statement follows  immediately from Lemma\;\ref{L:cc}.
\end{proof}
The next lemma is crucial in calculating the explicit action  $c_{o_i, e_i}(\gamma):$
\begin{lemma}\label {explicit}
$\forall i, j, i\neq j,  c_{o_i,e_i}(\sigma_i\sigma_j^{-1})=\zeta^{\nu_i}$
\end{lemma} 
\begin{proof}
We have
$\Theta_{o_i,e_i}(\sigma_i\sigma_j^{-1}e_j)=c_{o_i,e_i}(\sigma_i\sigma_j^{-1})\Theta_{o_i,e_i}(e_j)$.
Using Lemma\;\ref{ThetaSigma} we get
\begin{equation}
\Theta_{o_i,e_i}(\sigma_i\sigma_j^{-1}e_j)=\Theta_{o_i,e_i}(\sigma_i e_j)=\zeta^{\nu_i}\Theta_{o_i,e_i}(e_j).
\end{equation}
Conclude the result from the last expression. 
\end{proof}

We use the last results to calculate explicit images of the branch points and 
the action on the generators $\xi_{i,j}$.
\begin{theorem}\label{branchpointimages}
 $\forall \gamma\in \Gamma$ we have
\begin{enumerate}
\item $c_{o_0,e_0}(\gamma)=\zeta^{\nu_0}$,  \label{image1}
\item $c_{o_i,o_0}^p (\gamma) = c_{e_i,o_0}^p (\gamma)
=\prod_{j=1}^{p-1} c_{\varepsilon_{i,j}} (\gamma)$, where $\varepsilon_{i,j}=\sigma_i^j \sigma_0^{-j}$, \label{image2}
\item for $k=1$, \ldots, $p\,{-}\,1\ $\ $c_{o_i,e_i}(\xi_{j,k})=\zeta^{\nu_i\delta_{i,j}}$, where  $\delta_{i,j}$ is the Kronecker symbol. \label{image3}
\end{enumerate}
\end{theorem}
\begin{proof}
According to Lemma\;\ref{L:cc}, we have
$$ c_{o_0,e_0} (\sigma_0^{-1} \gamma \sigma_0) = 
c_{\sigma_0^{-1} o_0, \sigma_0^{-1} e_0} (\gamma) 
=  c_{o_0, e_0} (\gamma). $$
Because $\xi_{i,j}\,{=}\,\sigma_0 \xi_{i,j-1} \sigma_0^{-1}$ 
it is enough to show the assertion (\ref{image1}) for the generators $\xi_{i,1}$ only,  
this follows immediately from Lemma \ref{explicit}.

To show the second assertion we start with 
\begin{equation}
\Theta_{o_0,\sigma_i^j \sigma_0^{-j} o_0}(\gamma z)
= c_{\sigma_i^j \sigma_0^{-j}}(\gamma) \Theta_{o_0,\sigma_i^j  o_0}(z) 
= c_{o_0,\sigma_i^j o_0}(\gamma) \Theta_{o_0,\sigma_i^j  o_0} (z).
\end{equation}
where \eqref{OmegaGamma} is applied. At the same time,
\begin{equation}
c_{o_0,\sigma_i^j o_0}(\gamma) = c_{o_0,o_i}(\gamma) c_{o_i,\sigma_i^j o_0}(\gamma)
= c_{o_0,o_i}(\gamma) c^{-1}_{\sigma_i^j o_0,o_i}(\gamma) = c_{o_0,o_i}(\gamma) 
c^{-1}_{o_0,o_i}(\sigma_i^{-j} \gamma \sigma_i^{j}),
\end{equation}
where Lemma\;\ref{L:cc} is used.
Hence,
\begin{equation}
c_{o_0,o_i}(\sigma_i^{-j} \gamma \sigma_i^{j}) = \frac{c_{o_0,o_i}(\gamma)}{c_{\varepsilon_{i,j}}(\gamma)}.
\end{equation}
Because of part (2) of Proposition\;\ref{Pgamma0} we have
 $\prod_{i=0}^{p-1}c_{o_0,o_i}(\sigma_i^i \gamma \sigma_i^{-i})=1$, and so
\begin{equation}
1=\prod_{j=0}^{p-1} c_{o_0,o_i}(\sigma_i^{-j}\gamma \sigma_i^{j})
= c_{o_0,o_i}(\gamma) \prod_{j=1}^{p-1} c_{o_0,o_i}(\sigma_i^{-j} \gamma \sigma_i^{j})
= c^p_{o_0,o_i}(\gamma) \prod_{j=1}^{p-1} c_{\varepsilon_{i,j}}^{-1}(\gamma).
\end{equation} 
From this we conclude (\ref{image2}).

Now, we show \eqref{image3}. First, we verify this for $\xi_{i,1}$ and $\xi_{j,1},j\neq i.$ For $\xi_{i,1}$ 
the result follows from Lemma \ref{explicit}. For $\xi_{j,1}$, $j\neq i$, using its definition from Theorem\;\ref{T:sGenIntro},
we have 
\begin{equation}
c_{o_i,e_i}(\xi_{j,1})=c_{o_i,e_i}(\sigma_j\sigma_i^{-1}\sigma_i\sigma_0^{-1})=c_{o_i,e_i}(\sigma_j\sigma_i^{-1})c_{o_i,e_i}(\sigma_i\sigma_0^{-1})=\zeta^{\nu_i-\nu_i}=1.
\end{equation}
For a generator $\xi_{j,k}$, $k\,{>}\,1$, proceed by induction and obtain
\begin{equation}
c_{o_i,e_i}(\xi_{j,k})=c_{o_i,e_i}(\sigma_0\xi_{j,k-1}\sigma_0^{-1})=c_{o_i, e_i}(\xi_{i,1}^{-1}\sigma_i\xi_{j, k-1}\sigma_i^{-1}\xi_{i,1})=c_{o_i, e_i}(\sigma_i\xi_{j, k-1}\sigma_i^{-1}).
\end{equation}
Conclude \eqref{image3} from the induction hypothesis and Lemma \ref{L:cc}.
\end{proof}
\begin{remark}
Our results strengthen  the corresponding theorems in \cite{V1,V2}. We explicitly identify the multiplier for $\Theta_{o_i,e_i}(\sigma_i z),$ and rewrite the results from \cite{V1,V2} in terms of the generators $\varepsilon_{i,k}$. 
\end{remark}
 \section{Non-Special Divisors}
 We will combine the results from previous sections to produce generalizations of $\lambda$-function for cyclic Mumford curves generalizing \cite{T}. We reproduce the  theorem proved in \cite{V1} describing the equation for cyclic Mumford curves.
 \begin{theorem}
 Let $\sigma_i$, $i=0$, \ldots, $s$, be elements of order $p$,
 which serve as generators of $\Gamma_0$ from Theorem\;\ref{T:sGenIntro}.
Each  $\sigma_i$ is similar to a diagonal matrix 
$\beta_i \,{=}\, \diag(\zeta^{\nu_i }, 1)$.
Then the correcponding Mumford curve satisfies the  equation
\begin{equation}
y^p=\prod_{i=0}^s (x-a_i)^{r_i}(x-b_i)^{l_i}
\end{equation}
 such that
 \begin{enumerate}
 \item $\nu_i r_i=1\mod p$
 \item $-\nu_i l_i=1\mod p$
 \item $a_i,b_i$ are images of  $o_i$, $e_i,$ the fixed points of $\sigma_i$. 
 \end{enumerate}
 \end{theorem}
 \begin{proof}
 We prove the assertion for  $\sigma_0$. 
 Similar reasoning is applied to each $\sigma_i$. 
 Assume that the fixed points of $\sigma_0$ are $0$ and $\infty$. Define 
 $$\Theta_0(z)=\prod_{\gamma \in \Gamma_0}\frac{z-\gamma 0}{z- \gamma \infty}.$$ 
 Then the divisor of zeros  of $\Theta_0$ mapped to $\Omega/\Gamma_0$ coincides 
 with the divisor of zeros of the mapping $\phi \,{:}\, X_{\Gamma} \,{\to}\, X_{\Gamma_0}$, 
 which is the mapping given algebraically by $(x,y)\mapsto x\in \mathbb{P}^{1}(\mathbb{K})$.
 In this section we denote the mapping $\phi$ by $x$, and so 
$(x \circ \pi) (z) = \phi (\pi(z))$ for all $z\,{\in}\, \Omega$,
 where $\pi\,{:}\,\Omega \,{\to}\, X_{\Gamma}$.
 Conclude that
 \begin{equation}
 \Theta_0(z)= \epsilon\,  (x \circ \pi) (z).
 \end{equation}

 We can multiply $\Theta_0(z)$ by a constant if necessary, and assume that $\epsilon=1$. Then  
 \begin{equation}
\Theta_0(z)= \prod_{\gamma \in \Gamma_0}\frac{z-\gamma 0}{z- \gamma \infty}
=\prod_{k=0}^{p-1}\prod_{\gamma\in \Gamma}\frac{z-\gamma \sigma_0^k 0}{z-\gamma \sigma_0^k \infty}
=\prod_{k=0}^{p-1}\prod_{\gamma\in \Gamma}\frac{z-\gamma 0}{z-\gamma \infty}=\Theta^p_{0,\infty}(z).
 \end{equation}
 where $\Theta_{0,\infty}$ is defined in Definition \ref{ThetaABDef}.
 
 Let
 $$ F(z)= \prod_{i=1}^s
 \big(\Theta^p_{0,\infty}(z) - a_i \big)^{r_i} 
 \big(\Theta^p_{0,\infty}(z) - b_i \big)^{l_i}.$$
 
 Hence, for the function $y$ defined on $X_{\Gamma}$ we have that
 \begin{equation}
 (y\circ\pi)^p=\Theta_{0,\infty}^{p r_0} F(z)
 \end{equation}
 and so
\begin{equation}
(y\circ \pi)(z)=\Theta_{0,\infty}^{r_0}(z)F(z)^{1/p}.
 \end{equation}
Note that $F(\sigma_0z) \,{=}\, F(z)$ because  
$\Theta^p_{0,\infty}(\sigma_0 z)=\zeta^{p \nu_0}\Theta^p_{0,\infty}(z)=\Theta^p_{0,\infty}(z).$ 
On the other hand, $(x\circ \pi)(\sigma_0z)$,$(y \circ \pi)(\sigma_0z)$ 
induces an automorphism of order $p$ on $X_{\Gamma}$. Therefore,
\begin{multline}
\zeta (y\circ\pi)(z) = (y\circ \pi)(\sigma_0 z) = 
\Theta^{r_0}_{0,\infty}(\sigma_0 z) F(\sigma_0 z)^{1/p} \\
= \zeta^{\nu_0 r_0}\Theta^{r_0}_{0,\infty}(z)F(z)^{1/p} 
= \zeta^{\nu_0 r_0} (y\circ \pi)(z).
\end{multline}
Conclude $\nu_0 r_0 \,{=}\,1\modOp p$. Repeating the calculation for $\infty$, we find  that 
if $l_0$ is the order at $\infty$ we have ${-}\nu_0 l_0 \,{=}\,1\modOp p$. 
Similar results can be obtained for fixed points of $\sigma_i$, $i=1$, \ldots, $s$.
 \end{proof}
 
 Now explain how to produce non-special divisors for Mumford curves. 
 \begin{theorem}\label{ThetaNV}
 Assume that the fixed points of $\sigma_0$ are $0$, $\infty$ of the cyclic Mumford curve $X_{\Gamma}$
 defined by  the equation: $y^p=x^{r_0}\prod_{i=1}^s (x-a_i)^{r_i}(x-b_i)^{p-r_i}$. 
 Choose $d_j$, $j\,{=}\,0$, $1$ \ldots, $2s$,  
 such that 
  \begin{enumerate}
\item $0\leq d_j \leq (p-1)$,
 \item $d_0 +\sum_{i=1}^s (d_i + d_{i+s})=g$, where  $g=\frac{1}{2} (p-1)(2s-1)$ is the genus of $X_{\Gamma}$,
 \item for  $k=0$, \ldots, $p-1$
\begin{equation}
\quad g-\overline{d_j+kr_j}+\overline{\sum_{j=1}^{2s} kr_j}+1=0.
\end{equation}
 \end{enumerate}
 Let $D=d_0 o_0 +\sum_{i=1}^s d_i o_i+\sum_{i=1}^s d_{i+s} e_i$, where 
 $o_0\,{=}\,\phi^{-1}(0)$, and $o_i \,{=}\, \phi^{-1}(a_i)$, $e_i \,{=}\, \phi^{-1}(b_i)$, $i=1$, \ldots, $s$.
 Then $\theta_{\Gamma}(u_{\infty}(D-K_{\Gamma}))$ does not vanish.
 \end{theorem}
 \begin{proof}
 By Theorem\;\ref{dimensionformula}  $\dim(H^0(X,\mathcal{O}(-D))=0$, which means that 
 the divisor $D$ is non-special. Then, we apply Theorem\;\ref{RVT}. 
 \end{proof}
 \section{$\lambda$ Teitelbaum functions for Mumford cyclic curves}
 To obtain generalization of  $\lambda$-function, introduce
 \begin{definition}
 Define the function 
 \begin{equation}
 \theta[D]_{\Gamma}(c) = \theta_{\Gamma}(c-u_o(D-K_{\Gamma})),\quad c\in G_{\Gamma}.  
 \end{equation} 
 \end{definition}{ The function $ \theta[D]_{\Gamma}$ serves as an analogue
 of the theta function with a characteristic. }
 
 Let $D_1\,{=}\,\sum_{j}d_{j} B_j$, and
 $D_2\,{=}\,\sum_{j} \tilde{d}_{j} B_j$ be two divisors constructed from
 distinct branch points $B_j$ with $d_{j}$ and $\tilde{d}_j$, $j\,{=}\,0$, $1$, \ldots, $2s$, 
 satisfying the condition of Theorem~\ref{ThetaNV}.  
Consider the function  
$$\Lambda(P) = \frac{ \theta[D_1]_{\Gamma}(u_o(P))^p}
{ \theta[D_2]_{\Gamma}(u_o(P))^p},\quad P \in X_{\Gamma}.$$
It's a well defined function on $X_{\Gamma}$, because $u_o(D)$, $u_0(\tilde{D})$ 
are points of order $p$ in $G_{\Gamma}/A_{\Gamma}$. 
Thus, the divisor of zeros of this function is
$$p d_0 o_0 + p \sum_{i=1}^{s} \big( (d_{i} - \tilde{d}_i) o_i + (d_{i+s} - \tilde{d}_{i+s}) e_i \big). $$ 
$\forall Q_1,Q_2$ the divisor of  $\frac{x-\phi(Q_1)}{x-\phi(Q_2)}$ is: $p(Q_1-Q_2).$ 
Thus, $\Lambda$ descends to $h:\mathbb{P}^1(\mathbb{K}) \to \mathbb{P}^{1}(\mathbb{K})$, and  
$$\Lambda(P)=C\prod_{i=1}^s (\phi(P)-a_i)^{d_{i}- \tilde{d}_{i}} (\phi(P)-b_i)^{d_{i+s}- \tilde{d}_{i+s}}.$$
In order to obtain generalizations of  $\lambda$-function from \cite{T}, choose any $P_1$, $P_2\,{\in}\, X_{\Gamma}$. 
$\phi(P_1)$, $\phi(P_2)$  are allowed to be images of branch points 
 as long as $d_{i}$ and $\tilde{d}_{i}$ are equal for these branch points. 
Then
\begin{equation}
\frac{\Lambda(P_1)}{\Lambda(P_2)}
= \frac{ \theta[D_1]_{\Gamma}(u_o(P_1))^p \theta[D_2]_{\Gamma}(u_o(P_2))^p}
{ \theta[D_2]_{\Gamma}(u_o(P_1))^p \theta[D_1]_{\Gamma}(u_o(P_2))^p}.
\end{equation}

\section{Hyperelliptic case}
We analyze hyperelliptic curves in this section.

Assume that  $\sigma_0$, \ldots, $\sigma_g$ are generators of $\Gamma_0$ of order $2$,
and $g$ denotes the genus of a hyperelliptic curve. 
The fixed points of $\sigma_i$ are ordinary points $a_i$, $b_i$. Let $\xi_i = \sigma_i \sigma_0^{-1}=\sigma_i\sigma_0$, 
$i=1$, \ldots, $g$. We have the following: 
\begin{theorem}\
\begin{enumerate} 
\item $c_{o_0,e_0}(\xi_i)=-1$ for $i=0$, \ldots, $g$ \label{first}
\item $c_{o_i, e_0}^2=c_{e_i, e_0}^2=c_{\xi_i}$\label{second}
\item $c_{o_i, e_0}=c_{o_i, e_i}c_{e_i, e_0}$\label{third}
\item $c_{o_i,e_i}(\xi_j)=(-1)^{\delta_{ij}}$, $i,j\geq 1$, and $\delta_{ij}$ is the Kronecker delta. \label{fourth}
\end{enumerate} 
\end{theorem} 
 \begin{proof} 
 This follows from \ref{branchpointimages} when $p=2.$
 \end{proof}

We calculate explicitly the image of $K_{\Gamma}$. Choose a polarization such that $\rho_{\Gamma}(c_{\xi_i},c_{\xi_i})=c_{e_i,e_0}.$ For this polarization we form the theta function
\begin{equation} 
\theta_{\Gamma}(c) = \sum_{c_\gamma\in A_{\Gamma}}\sfa_{\Gamma,c_\gamma}(c),
\end{equation}
and $\sfa_{\Gamma,c_\gamma}$ is the cocycle associated with the polarization above. 
We have the following lemma that enables us to determine the zeros of the $\theta$ function we just defined. 
\begin{lemma}\label{coollemma} 
Let $c \in G_{\Gamma}$ such that  $c^2\,{=}\,c_\gamma \,{\in}\, A_{\Gamma}$ 
for $\gamma\,{\not\in}\, [\Gamma,\Gamma]$ and 
$c(\gamma)=-\rho_{\Gamma}(c_\gamma,c_\gamma).$ Then $\theta_{\Gamma}(c)=0$.
\end{lemma}
\begin{proof}
We have 
\begin{equation}
\theta_{\Gamma}(c)=\theta_{\Gamma}(c^{-1}c_\gamma)
= \sfa^{-1}_{\Gamma,c_\gamma}(c^{-1})\theta_{\Gamma}(c^{-1}).
\end{equation}  
At the same time, $\sfa^{-1}_{\Gamma,c_\gamma}(c^{-1})
= \rho_{\Gamma} (c_\gamma,c_\gamma)\times c(\gamma)^{-1}=-1$. 
Because $\theta_{\Gamma}(c)$ is an even function  we have  $\theta_{\Gamma}(c)=0$.
\end{proof} 
We apply Lemma\;\ref{coollemma} to $c_{e_i, e_0}$ and $c_{o_i, e_0}$ to obtain the following corollary.
\begin{corollary}\label{polarizationchoice}
Under the choice of the polarization $\rho_{\Gamma}(c_{\xi_i},c_{\xi_i})=c_{e_i, e_0},$ the zeros of  $\theta_{\Gamma}$ are the points $o_i.$
\end{corollary}
More generally, for any polarization we have $c_{e_i, e_0}\,{=}\,{-}c_{o_i, e_0}\,{=}\,{\pm} \rho_{\Gamma}\left(c_{\gamma_i},c_{\gamma_i}\right)$. Thus, the zeros of $\theta_{\Gamma,\rho}$ correspond to the points $o_i, e_i.$
Below, we  show that any divisor of the form $o_{i_1}\,{+}\, \cdots \,{+}\,o_{i_k}\,{+}\,e_{i_{k+1}}\,{+}\,\cdots \,{+}\,e_{i_g}$ with all distinct points is a non-special divisor. Combining this with the Riemann vanishing theorem we get the following 
\begin{theorem} 
Under the choice of polarization $\rho_{\Gamma}(c_{\xi_i},c_{\xi_i})=c_{e_i,e_0}, $  $K_{\Gamma}=\sum_{i=1}^g o_i$ in $J(\Gamma).$ 
\end{theorem} 
From now on we will work with the function associated with the polarization defined in Corrolary \ref{polarizationchoice}
\subsection{$p$-adic Lambda function for  hyper-elliptic curves} 
We calculate the degree of vanishing for any divisor $D$, $\deg D \,{=}\,g$, supported on the branch points $\{o_0$, \ldots, $o_g$, $e_0$, \ldots, $e_{g}\}$. We use that to obtain explicit formulas  for cross ratios of  $\phi(o_i)$, $\phi(e_i),$ generalizing expressions
from complex analysis. Below, we do not distinguish $o_o$, \ldots, $o_g$, $e_0$, \ldots, $e_g$, and denote them by $B_1$, \ldots $B_{2g+2}$. We assume that $B_1=\phi^{-1}(\infty)$ and $B_2=\phi^{-1}(0).$
\begin{theorem} \label{mmm}
Let $r,s$ be non-negative integers such that $2r+s=g.$ Then for every choice of branch points $B_{i_1}$, \ldots, $B_{i_r}$, $B_{j_1}$, \ldots, $B_{j_s}$ we have that $i(2\sum_{k=1}^r B_{i_k}\,{+}\,\sum_{l=1}^s B_{j_l})=s$.
\end{theorem} 
\begin{proof}
This follows  from Theorem\;\ref{dimensionformula}.
Let $D\,{=}\,2\sum_{k=1}^r B_{i_k} \,{+}\,\sum_{l=1}^s B_{j_l}$. The Riemann-Roch theorem implies 
\begin{equation}
i(D)=g-1- \deg D + r({-}D) = r({-}D)-1.
\end{equation}
Thus, it is  enough to show that $r({-}D)\,{=}\,s\,{+}\,1$. ${-}D$ is a divisor invariant under the action of the automorphism. Thus, we decompose  $\mathcal{O}({-}D)$ into a direct sum: $$\mathcal{O}(-D)=V_0\bigoplus V_1.$$  $V_0$, $V_1$ are eigenspaces with eigenvalues $\pm{1}.$  For  $f\,{\in}\, V_1,$ $f/y$ is isomorphic  to a subspace of  
$V_0(\mathcal{O}(-D+\divOp (y)))$ of functions that are fixed by the hyperelliptic involution. The mapping  
$D=\sum d_i B_i,\mapsto \sum d_i\phi(B_i)$,  $d_i\,{=}\,0$, $1$ induces an isomorphism between 
$V_0(\mathcal{O}(-D))$ and $\mathcal{O}(-D_0)$,  $D_0=\sum_{i} d_i\phi(B_i)$, where $B_i$ are the branch points. 
Now, apply the Riemann-Roch theorem to $\mathbb{P}^{1}(\mathbb{K})$ to conclude the result. 
\end{proof}
The following is an immediate consequence.
\begin{corollary} 
Let $D=\sum_{j=1}^{l} o_{i_j}+\sum_{k=1}^{m} e_{i_k}$, $l+m=g$, $i_j, i_k \neq 0$, then $D$ is a non-special divisor. 
\end{corollary}  
Consider two subsets $\mathbf{P_1}$, $\mathbf{P_2}$ of the set $\{1$, $2$, \ldots, $2g+1\}$ such that they  differ exactly in one element. For example, if we have a genus $2$ curve, we can choose
$\mathbf{P_1}=\{1,2\}$, $\mathbf{P_2}=\{1,3\}.$  $B_0$ is the basepoint of the mapping $u_{B_0}$.
Let $\mathbf{O}\,{=}\,\{3$, $5$, \ldots, $2g\,{+}\,1\}$, then 
$K_{\Gamma}\,{=}\, \sum_{i \,{\in}\,\mathbf{O}}^g B_{i}$. 
We  also require $\mathbf{P_1}\,{\neq}\,\mathbf{O}$, and $\mathbf{P_2}\,{\neq}\,\mathbf{O}$. 
Now, we define 
$$\theta_{\Gamma}[\mathbf{P_j}](c) = \theta_\Gamma 
\Big(c-u_{B_0}\Big(\textstyle\sum_{i\in \mathbf{P_j}} B_i+K_{\Gamma}\Big) \Big).$$
 Consider the mapping 
\begin{equation} 
\forall Q \in X_{\Gamma} \qquad
P\mapsto \frac{\theta^2_{\Gamma}\left[\mathbf{P_1}\right](u_{B_0}(Q))}
{\theta^2_{\Gamma}\left[\mathbf{P_2}\right](u_{B_0}(Q))}.
\end{equation} 
According to Theorem\;\ref{RVT}, the zero and poles of these functions are
$$
2 \Big(\sum_{i\in \mathbf{P_1}} B_i\Big)
- 2 \Big(\sum_{i\in \mathbf{P_2}} B_i \Big)= 2B_l - 2B_m,
$$
and $l$, $m$ are the unique elements which are different in $\mathbf{P_1}$ and $\mathbf{P_2}$. 
Hence, we have 
$$\frac{\theta^2_{\Gamma}\left[\mathbf{P_1}\right](u_{B_0}(Q))}
{\theta^2_{\Gamma}\left[\mathbf{P_2}\right](u_{B_0}(Q))} = C\frac{\phi(P) - a_l}{\phi(P) - a_m}$$\label{hypc}
To cancel the constant $C$ we choose $Q_1,Q_2$ and $\phi(Q_i)\neq a_l$ and $\phi(Q_i)\neq a_m$ for $i=1,2.$ Take the ratio between the equalities \ref{hypc} we obtain: 
\begin{equation}
\frac{\theta^2_{\Gamma}\left[\mathbf{P_1}\right](u_{B_0}(Q_1))\theta^2_{\Gamma}\left[\mathbf{P_2}\right](u_{B_0}(Q_2))}{\theta^2_{\Gamma}\left[\mathbf{P_1}\right](u_{B_0}(Q_2))\theta^2_{\Gamma}\left[\mathbf{P_2}\right]((u_{B_0}(Q_1)}
=\frac{(\phi(Q_1) - a_l)(\phi(Q_2) - a_m)}{(\phi(Q_1) - a_m)(\phi(Q_2)-a_l)}
\end{equation}
Choosing $Q_i$ to be branch points that aren't equal to $a_k, a_l$ we obtain: 
\begin{theorem} 
For any subsets $\mathbf{P_1}$, $\mathbf{P_2}$ as above let the pre-images of $a_l$, $a_m$ be the unique points such that $l\,{\not\in}\, \mathbf{P_2}$ and $m\,{\not\in}\, \mathbf{P_1}$. 
Assume $a_k$, $a_j$ any points such that $k$, $j\,{\not\in}\, \mathbf{P_1}\cup\mathbf{P_2}$. We have 
\begin{equation}
\frac{\theta^2_{\Gamma}\left[\mathbf{P_1}\right](u_{B_0}(\phi^{-1}(a_k))\theta^2_{\Gamma}\left[\mathbf{P_2}\right](u_{B_0}(\phi^{-1}(a_j))}{\theta^2_{\Gamma}\left[\mathbf{P_1}\right](u_{B_0}(\phi^{-1}(a_j))\theta^2_{\Gamma}\left[\mathbf{P_1}\right]((u_{B_0}(\phi^{-1}(a_l)}
=\frac{(a_k - a_l)(a_j - a_m)}{(a_k - a_m)(a_j-a_l)}\end{equation}
\end{theorem} 
\section{conclusion}
In this note we generalized the classical formulas for the expression of the cross ratio of the branch points of cyclic and hyper-elliptic curves through theta functions to Mumford curves. Previously these formulasThese formulas were known only for $g=2$ \cite{T}. In our approach we replaced the fundamental domain considerations of the action of $PGL_2(K)$  with direct computation as in \cite{V1}. Similar ideas should generalize these formulas to any Mumford curves that are Galois extensions of $\mathbb{P}^{1}(K).$ We intend to pursue these questions in subsequent work. 
\bibliographystyle{plain}

\begin{thebibliography}{11}   
\bibitem[DS]
{DS75}
Dyer, J.L, Scott P \emph{Periodic automorphisms of free groups} Comm. in Algebra. \textbf{3} (1975) pp. 195-201

\bibitem[EG06]
{EG06} 
T. Grava, V.Enolskii \emph{ Thomae type formulae for singular $Z\sb N$ curves} Lett. Math. Phys. \textbf{76} (2006), no. 2-3, 187--21

\bibitem[ER08]
{er08} V. Z. Enolski and  P. Richter, \emph{ Periods of hyperelliptic
integrals expressed in terms of $\theta$-constants by means of
{T}homae formulae,}
 Philos. Lond. Trans. R. Soc. Ser. A Math. Phys. Eng. Sci. \textbf{ 366} (2008)
 no. 1867, 1005-1024.

\bibitem [FZ10]
{FZ10} H., Farkas, S.Zemel \emph{Generalizations of Thomae's Formula for Zn Curves } Progress in mathematics


\bibitem[FK980]{FK980}
Farkas, H. M. and Kra, I.,\emph{ Riemann Surfaces} Springer, New York, 1980.


\bibitem[G]{G}
L.Gerritzen, \emph{On Non Archimedean Representations of Abelian Varieties,} Math. Annalen,\textbf{275}(1972) pp.323-346


\bibitem[Fay973]
  {Fay973} J.~D. Fay, \emph{Theta functions on {R}iemann surfaces},
  Lectures Notes in Mathematics (Berlin), vol. 352, Springer, 1973.

\bibitem[Fay979]
  {Fay979} J.~D. Fay,
Fay, J. D. 1979 \emph{On the Riemann–Jacobi formula}, Nachrichten der Akadedemie der Wissenschaften
in G\"ottingen. II. Mathematisch-Physikalische Klasse 5, (1979) 61–73.

\bibitem[GV]
{GV} L. Gerritzen,M.Van Der Put \emph{Schottky Groups and Mumford Curves }, 
Lecture Notes in Mathematics(Berlin), vol.817, Springer 1980 


\bibitem[Kle886]
  {kl86} F.~Klein, \emph{{\"U}ber hyperelliptische {S}igmafunctionen},
  Math. Ann.  \textbf{27} (1886), 431--464.

\bibitem[Kle888]
  {kl88} F.~Klein, \emph{{\"U}ber hyperelliptische {S}igmafunctionen},
  Math. Ann.  \textbf{32} (1888), 351--380.

\bibitem[Kop10]{
Kop10}Y.Kopeliovich \emph{Thomae Formula for General Cyclic Covers of $CP^{1}$} 
Letter of Mathematical of Physics \textbf{94} (2010), 313-333


\bibitem[MP08]{MP08}
 S. Matsutani and E. Previato, \emph{Jacobi inversion on strata of the Jacobian of the $C_{rs}$ curve},
 J. Math. Soc. Japan \textbf{60}, 4, (2008), 1009-1044.

 \bibitem[MP10]{MP10}
 S. Matsutani and E. Previato,
 \emph{Jacobi inversion on strata of the Jacobian of the $C_{rs}$ curve II},
 arXiv: 1006.1090 [math.AG]

\bibitem [M2]{M2} D.Mumford \emph{Tata Lecture on Theta Vol.II }, Birkhauser 2007

\bibitem[Nak997]{nak997} A. ~Nakayashiki \emph{On the Thomae formula for $Z_N$ curves} Publ. Res. Inst. Math Sci. \textbf{33} (1997) 987-1015

\bibitem[Nak08a]
  {Nak08} A.~Nakayashiki, \emph{{A}lgebraic {E}xpression of {S}igma
    {F}unctions of $(n,s)$ {C}urves}, arXiv:0803.2083, 2008.

\bibitem[KT10]
{KT10} Matsumoto Keiji,Terasoma Tomohide
 \emph{Degenerations of triple covering and Thomae's formula},  arXiv1001.4950M, 2010

\bibitem[Ros851]
{ros851}
G.~Rosenhain.
\newblock Abhandlung {\"u}ber die {F}unktionen zweier {V}ariabler mit vier
  {P}erioden.
\newblock {\em M{\'e}m. pr{\'e}s. l'Acad. de Sci. de France des savants},
  XI:361--455, 1851.
\newblock The paper is dated 1846. German Translation: H. Weber (Ed.),
  Engelmann-Verlag, Leipzig 1895.
\bibitem[T]
{T} Teitelbaum Jeremy
\emph{$p$-adic periods of genus two Mumford Schottky Curves},
J.Reine Angew. Math, \textbf{385} (1988) 117-151


\bibitem[Tho870]
	{Tho870} Thomae C.J \emph{ Beitrag zur Bestimmung
 $\theta (0, 0, ... 0)$ durch die Klassenmuduln algebraicher Funktionen} 
J.Reine Angew. Math, \textbf{71}(1870) 201-222

\bibitem[V1]{V1}Vansteen G. \emph{Note on Coverings of the Projective Line by Mumford Curves}, Belgian of Mathematical Society, \textbf{38}(1983) series B.31-38

\bibitem[V2]{V2} Vansteen G. \emph{Galois coverings of the Non-Archimedean Projective Line},  Mathematische Zeitschrift, \textbf{180}(1982) 217-224

\bibitem[V3]{V3} Vansteen G.\emph{The Schottky-Jung Theorem for Mumford curves}, Annales de Institut  Fourier, \textbf{39}(1989) 1-15


\end{thebibliography}

\end{document}